\documentclass[12pt,reqno]{amsart}
\usepackage{graphicx}
\usepackage{amssymb,amsmath}
\usepackage{amsthm}
\usepackage{color,graphicx}
\usepackage{hyperref}
\usepackage{color}
\usepackage{epstopdf}

\newcommand{\be}{\begin{equation}}

\newcommand{\ee}{\end{equation}}

\newtheorem{obs}{Remark}[section]

\newcommand{\sn}{{\rm \,sn}}
\newcommand{\dn}{{\rm \,dn}}

\newcommand{\R}{{\mathbb R}}

\newcommand{\ve}{{\varepsilon}}

\numberwithin{equation}{section}
\numberwithin{figure}{section}

\newtheorem{theorem}{Theorem}[section]
\newtheorem{proposition}[theorem]{Proposition}
\newtheorem{remark}[theorem]{Remark}
\newtheorem{lemma}[theorem]{Lemma}
\newtheorem{corollary}[theorem]{Corollary}
\newtheorem{definition}[theorem]{Definition}

\begin{document}
\vglue-1cm \hskip1cm
\title[Periodic Waves for a dispersive equation]{Periodic Traveling-wave solutions for regularized dispersive equations: Sufficient conditions for orbital stability with applications}

\begin{center}

\subjclass[2000]{76B25, 35Q51, 35Q53.}

\keywords{Orbital stability, regularized dispersive equation, periodic waves}

\maketitle

{\bf Fabr\'icio Crist\'ofani }

{IMECC-UNICAMP\\
	Rua S\'ergio Buarque de Holanda, 651, CEP 13083-859, Campinas, SP,
	Brazil.}\\
{ fabriciocristofani@gmail.com}

\vspace{3mm}

{\bf F\'abio Natali}

{Departamento de Matem\'atica - Universidade Estadual de Maring\'a\\
	Avenida Colombo, 5790, CEP 87020-900, Maring\'a, PR, Brazil.}\\
{fmanatali@uem.br}

\vspace{3mm}

{\bf Ademir Pastor}

{IMECC-UNICAMP\\
	Rua S\'ergio Buarque de Holanda, 651, CEP 13083-859, Campinas, SP,
	Brazil.}\\
{apastor@ime.unicamp.br}

\end{center}

\begin{abstract}
In this paper, we establish a new criterion for the orbital stability of
periodic waves related to a general class of regularized dispersive equations. More specifically, we present sufficient conditions for the stability without knowing the positiveness of the associated hessian matrix. As application of our method, we show the orbital stability for the fifth-order model. The orbital stability of periodic waves resulting from a minimization of a convenient functional is also proved.
\end{abstract}

\section{Introduction}

We present sufficient conditions for the orbital stability of periodic traveling-wave solutions associated to the regularized dispersive model
\be\label{rDE}
u_t+u_x+uu_x+(\mathcal{M}u)_t=0,
\ee
where $u:\R\times\R\to\R$ is a real spatially $L$-periodic function. Here $\mathcal{M}$ is a differential or pseudo-differential operator in the periodic setting and it is defined as a Fourier multiplier  by 
\begin{equation*}
\widehat{\mathcal{M}g}(\kappa)=\theta(\kappa)\widehat{g}(\kappa), \quad \kappa\in\mathbb{Z}.
\end{equation*}
The symbol $\theta$ is assumed to be even and continuous on $\R$ satisfying
\begin{equation}\label{A1A2}
\upsilon_1|\kappa|^{m_1}\leq\theta(\kappa)\leq \upsilon_2|\kappa|^{m_1}, \quad m_1>1/3,
\end{equation}
for all $\kappa\in\mathbb{Z}$ and for some $\upsilon_i>0$, $i=1,2$.

Regularized equations appear as alternative models to describe the propagation of nonlinear waves in several physical contexts. Indeed, if $\mathcal{M}=-\partial_x^2$, equation $(\ref{rDE})$  reduces to the so called BBM equation,
\begin{equation}\label{kdv1}
u_t+u_x+uu_x-u_{xxt}=0,
\end{equation}
which was originally derived by Benjamin-Bona-Mahony \cite{bbm} 
 as an alternative model to
 the well known Korteweg-dr Vries  equation for small-amplitude, long wavelength surface
 water waves.  Also,
if  $\mathcal{M}=\mathcal{H}\partial_x$, equation $(\ref{rDE})$ reduces to the regularized Benjamin-Ono equation 
\begin{equation}\label{BO1}
u_t+u_x+uu_x+\mathcal{H}\partial_x u_t=0,
\end{equation}
where $\mathcal{H}$ indicates the Hilbert transform defined via its Fourier transform as
$$\widehat{\mathcal{H}f}(\kappa)=-i\mbox{sgn}(\kappa)\widehat{f}(\kappa),\ \ \ \ \ \ \ \ \ \kappa\in\mathbb{Z}.$$
Equation \eqref{BO1} models  the evolution of long-crested waves
at the interface between two immiscible fluids. It also appears in the two-layer system 
created by the inflow of fresh water from a river into the sea (see \cite{kalisch} and references therein). For the orbital stability of periodic traveling waves for \eqref{BO1} we refer the reader to \cite{ASB}.

Formally, equation $(\ref{rDE})$ admits the conserved quantities
\begin{equation}\label{Eu}
	P(u)=\frac{1}{2}\int_{0}^{L}\Big(u\mathcal{M}u-\frac{1}{3}u^3\Big)dx,
\end{equation}
\begin{equation}\label{Fu}
	F(u)=\frac{1}{2}\int_{0}^{L}\left(u\mathcal{M}u+u^2\right)dx,
\end{equation}
and 
\begin{equation}\label{Mu}
	M(u)=\int_{0}^{L}u\,dx.
\end{equation}

A traveling wave solution for \eqref{rDE} is a solution of the form $u(x,t)=\phi(x-\omega t)$, where $\omega$ is a real constant representing the wave speed and $\phi:\R\to\R$ is a periodic function. Substituting this form into (\ref{rDE}), we obtain
\begin{equation}\label{ode-wave}
	\omega \mathcal{M}\phi+(\omega-1)\phi-\frac{1}{2}\phi^2+A=0,
\end{equation}
where $A$ is a constant of integration.

In view of the conserved quantities (\ref{Eu})-(\ref{Mu}) we may define the augmented Lyapunov functional,
\begin{equation}\label{lyafun}
	G(u)=P(u)+(\omega-1) F(u)+AM(u),
\end{equation}
and the linearized operator around the wave $\phi_{(\omega,A)}$,
\begin{equation}\label{operator}
	\mathcal{L}:=G''(\phi)=\omega\mathcal{M}+(\omega-1)-\phi.
\end{equation}
Note in particular that $G'(\phi)=0$.
Thus, it is expected that the functional $G$ defined in (\ref{lyafun}) plays a crucial role in order to guarantee the orbital stability.

Let us connect our work with the current literature. First of all, since the operator $\mathcal{M}$ satisfies the general relation \eqref{A1A2}, we are able  to address in a unified way a large number of dispersive models. However, our main motivation come from the results for the generalized BBM equation. Indeed, based on the work \cite{johnson}, the author in \cite{johnsonBBM} established sufficient conditions for the modulational/orbital stability of periodic waves related to the generalized BBM equation
\begin{equation*}
u_t+u_x+u^pu_x-u_{xxt}=0,
\end{equation*}
where $p\geq1$ is an integer. In particular, if $1\leq p<4$, it was showed that the periodic waves in the solitary wave limit are stable (modulationally and nonlinearly). On the other hand, if $p>4$, the instability was established provided the corresponding wave speed $\omega$ is greater than a critical speed $\omega(p)>1$. To this end, the author has
constructed smooth periodic waves $\phi(\cdot,A,B,\omega)$, where the period $L$ depends smoothly on the triple $(A,B,\omega)\in\Omega\subset\mathbb{R}^3$. Here $B$ is the integration constant which appears in the quadrature form associated with the second order differential equation  $(\ref{ode-wave})$ with $\mathcal{M}=-\partial_x^2$. So, by assuming that the signal of the Jacobian matrices $L_{B}$, $\{L,M\}_{A,B}$ and $\{L,M,F\}_{A,B,\omega}$ are positive at the point $(A_0,B_0,\omega_0)\in\Omega\subset\mathbb{R}^3$, one has the orbital stability of the   waves $\phi(\cdot,A_0,B_0,\omega_0)$. Here,  
$$
\{f_1,\ldots, f_n\}_{x_1,\ldots,x_n}=\det 
\begin{pmatrix}
\dfrac{\partial(f_1,\ldots,f_n)}{\partial(x_1,\ldots,x_n)}
\end{pmatrix}.
$$
In the case $p=1,2,4$, the reader will also find some results in  \cite{ASB}, \cite{ASB1}, where the authors studied the orbital stability of some explicit solutions.

If $\mathcal{M}$ is the fractional derivative operator  $\mathcal{M}=(\sqrt{-\partial_x^2})^{\alpha}$, $1/3<\alpha\leq2$, in Fourier sense (the cases $\mathcal{M}=-\partial_x^2$ and $\mathcal{M}=\mathcal{H}\partial_x$ are included in that approach), in  \cite{hur}, the authors established the existence of minimizers for the energy functional. In addition, it has been proved that the local minimizers are orbitally stable provided that the determinant $\{F,M\}_{A,\omega}$ is assumed to be non-zero.

Our main goal in this paper is to establish a new criterion for the orbital stability where it is not necessary to know the positiveness of the associated Hessian matrix neither the Jacobians as determined in \cite{ABS}, \cite{hur} and \cite{johnsonBBM}. To do so, instead of considering $G$ as a Lyapunov function, based on the works \cite{ANP}, \cite{NP1} and \cite{Stuart}, we consider the new functional given by 
$$V(u)=G(u)-G(\phi)+N(Q(u)-Q(\phi))^2,$$ 
where $N$ is a positive constant to be determined later and $Q(u):=x_0F(u)+y_0M(u)$ with $x_0,y_0\neq0$ real constants also to be determined properly. This new functional removes the assumption of the mentioned positiveness in the stability theorem.

Next, we present a brief outline of our work. We will assume the following assumption:\\ 

\begin{enumerate}
	\item[\textit{(H)}] Assume $m_1>1/3$. Let $(\omega_0,A_0)\in\mathbb{R}\backslash\{0\}\times\R$ be fixed. Suppose that $\phi:=\phi_{(\omega_0,A_0)}\in X:=H_{per}^\frac{m_1}{2}([0,L_0])$ is an even periodic solution of (\ref{ode-wave}) in the sense of distributions with fixed period $L_0>0$. Moreover, assume the self-adjoint operator 
	\begin{equation}\label{L0}
	\mathcal{L}_0:=\mathcal{L}_{(\omega_0,A_0)}=\omega_0\mathcal{M}+(\omega_0-1)-\phi_{(\omega_0,A_0)}
\end{equation}
 has only one negative eigenvalue which is simple and zero is a simple eigenvalue whose eigenfunction is $\phi'$.
\end{enumerate}
\vspace{0.5cm}

Here and throughout the paper, $H_{per}^s([0,L_0])$ stands for the periodic Sobolev space of order $s\in\mathbb{R}$. When $s\geq0$, $H_{per,e}^s([0,L_0])$ indicates the subspace of $H_{per}^s([0,L_0])$ constituted by the even periodic functions.

The spectral properties of operator $\mathcal{L}_0$ in  assumption $(H)$ are crucial to obtain our results. In general, such properties are not easily obtained and one needs to work with the structure of the equation in hand to obtain them. However, there are some theories in the literature where we may get $(H)$. Indeed,
in many situations when $\mathcal{M}$ is a second order differential operator and $\phi$ is given in terms of the Jacobian elliptic functions, $\mathcal{L}_0$ turns out to be a Hill's operator with a Lam\'e type potential (see \cite{Magnus}). In particular, studying the spectrum of $\mathcal{L}_0$ is equivalent to studying the eigenvalue problem 
\begin{equation}\label{specproblem}
	\left\{\begin{array}{l}
		h''(x)+\left[\lambda-n(n+1)\cdot k^2 \sn^2\left(x,k\right)\right]h(x)=0,\\
		h(0)=h(2K(k)), \quad h'(0)=h'( 2K(k)),
	\end{array}\right.
\end{equation}
where $\lambda$ is a real parameter and $n$ is a non-negative integer. Depending on $n$, the first eigenvalues of \eqref{specproblem} are well known (see e.g., \cite{ince}). Many applications using this approach have appeared in the literature (see e.g., \cite{angulo4} and references therein).
Another approach to obtain $(H)$ was given in \cite{Neves1}.  Assume that $\mathcal{M}$ is a second order differential operator. Recall from Floquet's theorem (see e.g., \cite{Magnus} page 4) that if $y$ is any solution of $\mathcal{L}_0 y=0$, linearly independent of $\phi'$, then there exists a constant $\theta$ satisfying 
\begin{equation}\label{thetadef}
	y(x+L)=y(x)+\theta \phi'(x).
\end{equation}
In particular, if $y$ satisfies the initial condition  $y'(0)=0$ then
by taking the derivative with respect to $x$ in both sides of \eqref{thetadef} and evaluating the result at $x=0$, we see that
\begin{equation}\label{thetadef1}
	\theta=\dfrac{y'(L)}{\phi''(0)}.
\end{equation}
Under these conditions Theorem 3.1 in \cite{Neves1} (see also \cite{NN}) states that the second eigenvalue of $\mathcal{L}_0$ is simple if and only if $\theta\neq0$; in addition, it is zero if and only if $\theta<0$. 
Finally, let us recall the approach given in \cite{AN}, which is based on the total positivity theory (see \cite{kar}) and can be applied to  local or nonlocal operators.  To give the precise statement of the result, we recall that a sequence $\{\alpha_n\}_{n\in\mathbb{Z}}$ of real numbers is said to be in the class $PF(2)$ discrete if
\begin{itemize}
	\item[(i)] $\alpha_n>0$, for all $n\in\mathbb{Z}$;
	\item[(ii)] $\alpha_{n_1-m_1}\alpha_{n_2-m_2}-\alpha_{n_1-m_2}\alpha_{n_2-m_1}>0$, for $n_1<n_2$ and $m_1<m_2$.
\end{itemize}
Theorem 4.1 in \cite{AN} states if $\phi$ is positive, even, and  $\widehat{\phi}>0$ and $\widehat{\phi^2}$ belong to the class $PF(2)$ discrete, then $\mathcal{L}_0$ satisfies $(H)$ (see also Section \ref{applic} below). Here $\widehat{f}$ stands for the Fourier transform of $f$.

With hypothesis $(H)$ in hand, we are enabled to construct a smooth surface
\begin{equation*}
(\omega,A)\in \mathcal{O}\mapsto \phi_{(\omega,A)} \in H_{per,e}^n([0,L_0]), \quad  n\in\mathbb{N},
\end{equation*}
 of periodic solutions for  $(\ref{ode-wave})$, with a fixed period $L_0$. This means that for any
 $(\omega,A)$ in the open neighborhood $\mathcal{O}$ of $(\omega_0,A_0)$, $\phi_{(\omega,A)}$ is a solution of \eqref{ode-wave} with period $L_0$. In addition, assumption $(H)$ is also suitable to obtain the non-positive spectrum
 of the linearized operator $\mathcal{L}_{(\omega,A)}$ in $(\ref{operator})$, since one has the convergence $\mathcal{L}_{(\omega,A)}\rightarrow\mathcal{L}_{(\omega_0,A_0)}$ in the sense of Kato (see detailed arguments in \cite{kato1}). As a consequence, we may prove the orbital stability of periodic waves without knowing the behavior of the Hessian matrix associated to the function $(\omega,A)\mapsto G(\phi_{(\omega,A)})$, as required in \cite{AN}, \cite{ASB}, \cite{grillakis1}, \cite{hur}, \cite{johnsonBBM}, \cite{NN}, and related references. 
Instead, assuming that assumption $(H)$ occurs and  $\omega_0-1-2A_0\neq0$, our orbital stability criterion is based in proving that the quantity $s(\phi)$ defined as

\begin{equation}
\label{sphi123}
\begin{split}
s(\phi)&=\displaystyle\left(2\omega_0(\omega_0-1)+2A_0+1\right)M(\phi)+\omega_0\int_{0}^{L_0}\phi \mathcal{M}\phi dx\\
&\quad+\displaystyle\left(2A_0(\omega_0+1)-\omega_0+1\right)L_0.
\end{split}
\end{equation}
is strictly positive. In fact, we have the following result
\begin{theorem}\label{coro14}
	Assume that assumption $(H)$ holds and let $s(\phi)$ be defined as in \ref{sphi123}. If $\omega_0-1-2A_0\neq0$ and $s(\phi)>0$, then the periodic wave $\phi$ is orbitally stable in $X$.
\end{theorem}

 In order to prove Theorem $\ref{coro14}$, we employ the recent developments in \cite{CNP} and \cite{NP1}, which are extensions of the approaches in  \cite{bona2}, \cite{grillakis1}, and \cite{johnson}  adapted to the periodic case.

 Our paper is organized as follows. In next section we present the existence of periodic waves related to  equation $(\ref{ode-wave})$, the behaviour of the non-positive spectrum of $\mathcal{L}$, and the orbital stability theory of periodic waves. The sufficient condition for the orbital stability of periodic waves is presented in Section \ref{suffstab}. Finally,  Section \ref{applic} is devoted to some applications of our theory.

\section{Orbital Stability of Periodic Waves}\label{OSPW}

In this section, we present our stability result. The main result of the section is Theorem \ref{teoest} which gives a criterion for the orbital stability.
Before stating the result itself, we need some preliminary tools. For functions $u$ and $v$ in $X:=H_{per}^\frac{m_1}{2}([0,L_0])$ we   let $\rho$ be the ``distance'' between $u$ and $v$  defined  by
\begin{equation*}
	\rho(u,v)=\inf_{y\in\mathbb{R}}||u-v(\cdot+y)||_{X}.
\end{equation*}
Roughly speaking the distance between $u$ and $v$ is measured through the distance between $u$ and the orbit of $v$, generated by translations.

Throughout this section we let $\phi:=\phi_{(\omega_0,A_0)}\in X$ be the periodic wave given in $(H)$.
Our precise definition of orbital stability is given below.
\begin{definition}\label{defstab}
	We say that an $L_0$-periodic solution $\phi$ is orbitally stable in $X$, by the periodic flow of \eqref{rDE},  if for any $\ve>0$ there exists $\delta>0$ such that for any $u_0\in X$ satisfying $\|u_0-\phi\|_X<\delta$, the solution $u(t)$ of \eqref{rDE} with initial data $u_0$ exists globally and satisfies
	$$
	\rho(u(t),\phi)<\ve,
	$$
	for all $t\geq0$.
\end{definition}

\begin{remark}
	The notion of orbital stability prescribes the existence of global solutions. Since questions of (local and) global well-posedness is out of the scope of this paper, we will assume  the periodic Cauchy problem associated with (\ref{rDE}), namely, 
	\begin{equation}\nonumber
		\left\{\begin{array}{llllll}
			u_t+u_x+uu_x+(\mathcal{M}u)_t=0, \\
			u(x,0)=u_0(x), \ \ \ x\in[0,L].
		\end{array}\right.
	\end{equation}
	is globally well-posed in $X$.
\end{remark}

For a given $\varepsilon>0$, we define the $\varepsilon$-neighborhood of the orbit $O_\phi=\{\phi(\cdot+y), y\in\R\}$ as
\begin{equation*}
	U_{\varepsilon} := \{u\in X;\ \rho(u,\phi) < \varepsilon\}.
\end{equation*}
In what follows, we set  
\begin{equation*}
\Upsilon_0=\{u\in X;\ \langle Q'(\phi),u\rangle=0\},
\end{equation*}
where $\langle\cdot,\cdot\rangle$ denotes the scalar product in $L^2_{per}([0,L_0])$.
Note that $\Upsilon_0$ is nothing but the tangent space to $\{u\in X; Q(u)=Q(\phi)\}$ at $\phi$. With these notations, our main theorem reads as follows.

\begin{theorem}\label{teoest} Suppose that assumption (H) holds. Moreover, for $\mathcal{L}_0$ defined in \eqref{L0}, assume the existence of $\Phi\in X$ such that $\langle\mathcal{L}_0\Phi,\varphi\rangle=0$, for all $\varphi\in \Upsilon_0$, and $\mathcal{I}=\langle\mathcal{L}_0\Phi,\Phi\rangle<0$, then $\phi$ is orbitally stable in $X$ by the periodic flow of   $(\ref{rDE})$.
\end{theorem}

In order to prove Theorem \ref{teoest} we follow the strategy put forward in \cite{CNP}, \cite{NP1}, and \cite{Stuart}. Let us start by showing that $\mathcal{L}_0$ is strictly positive when restricted to the space $\Upsilon_0\cap \{\phi'\}^\perp$.

\begin{lemma}\label{prop2}
		Under assumptions of Theorem \ref{teoest}, there exists a constant $c>0$ such that
	$$\langle\mathcal{L}_0v,v\rangle\geq c||v||_{X}^2,$$
	for all $v\in \Upsilon_0\cap \{\phi'\}^\perp$, where $\left\{\phi'\right\}^{\perp}:=\left\{u\in X ; \langle\phi',u\rangle=0\right\}$.
\end{lemma}
\begin{proof}
	See Proposition 4.12 in \cite{CNP}.
\end{proof}

Lemma \ref{prop2} is useful to establish  the following result.

\begin{lemma}\label{lemma1}
	Under assumptions of Theorem \ref{teoest}, there exist $N>0$ and $\tau>0$ such that
	$$\langle\mathcal{L}_0v,v\rangle +2N\langle Q'(\phi),v\rangle^{2}\geq \tau||v||_{X}^2,$$
	for all $v\in \left\{\phi'\right\}^{\perp}$.
\end{lemma}
\begin{proof}
Given $v\in \left\{\phi'\right\}^{\perp}$, define
	$$z=v-\zeta w.$$
where $w=\frac{Q'(\phi)}{||Q'(\phi)||_{L^2_{per}}}$ and $\zeta=\langle v,w\rangle$. Because  $\langle Q'(\phi),\phi'\rangle=0$, it is easily seen that $z\in\Upsilon_0\cap \left\{\phi'\right\}^{\perp}$. Thus, Lemma \ref{prop2} implies
\begin{equation}\label{eq01}
		\langle\mathcal{L}_0v,v\rangle \geq \zeta^2\langle\mathcal{L}_0w,w\rangle + 2\zeta\langle\mathcal{L}_0w,z\rangle +c||z||_X^2.
	\end{equation}
	Using Cauchy-Schwartz and Young's inequalities, we have
	\begin{equation}\label{eq02}
		2\zeta\langle\mathcal{L}_0w,z\rangle \leq \frac{c}{2}||z||_X^2 + \frac{2\zeta^2}{c}||\mathcal{L}_0w||_X^2.
	\end{equation}
	Furthermore, we may choose $N>0$ such that
	\begin{equation}\label{eq03}
		\langle\mathcal{L}_0w,w\rangle -\frac{2}{c}||\mathcal{L}_0w||_X^2 +2N||Q'(\phi)||_{L^2_{per}}^2\geq \frac{c}{2}\|w\|_{X}^2.
	\end{equation}
	We point out that $N$ depends only on $\phi$.
	Therefore, using (\ref{eq01}), (\ref{eq02}) and (\ref{eq03}), we conclude
	\begin{eqnarray}
		\langle\mathcal{L}_0v,v\rangle +2N\langle Q'(\phi),v\rangle^{2}&=&\langle\mathcal{L}_0v,v\rangle + 2N\zeta^2||Q'(\phi)||_{L^2_{per}}^2 \nonumber \\
		&\geq&\frac{c}{2}(\zeta^2\|w\|^2_X+||z||_X^2) \nonumber \\
		&\geq& \tau||v||_{X}^2. \nonumber
	\end{eqnarray}
	The proof is thus completed.
\end{proof}

Let $N>0$ be the constant obtained in the previous lemma. We define the functional $V:X\rightarrow\R$ as
\begin{equation}\label{functionalV}
	V(u)=G(u)-G(\phi)+N(Q(u)-Q(\phi))^2,
\end{equation}
where $G$ is the augmented functional defined in (\ref{lyafun}) with $(\omega,A)=(\omega_0,A_0)$. It is easy to see from $(\ref{functionalV})$ and $(\ref{ode-wave})$ that $V(\phi)=0$ and $V'(\phi)=0$.

\begin{lemma}\label{lemma2}
	There exist $\alpha>0$ and $D>0$ such that 
	$$V(u)\geq D\rho(u,\phi)^2$$
	for all $u\in U_{\alpha}$.
\end{lemma}
\begin{proof}
	First, note that from the definition of $V$ it follows that
	$$\langle V''(u)v,v\rangle=\langle G''(u)v,v\rangle+2N(Q(u)-Q(\phi)) \langle Q''(u)v,v\rangle+2N\langle Q'(u),v\rangle^2,$$
	for all $u,v\in X$. In particular, 
	$$\langle V''(\phi)v,v\rangle=\langle \mathcal{L}_0v,v\rangle+2N\langle Q'(\phi),v\rangle^2.$$
	Consequently, from Lemma \ref{lemma1} we get
	\begin{equation} \label{eq001}
	\langle V''(\phi)v,v\rangle\geq\tau||v||_X^2,
	\end{equation}
	for all $v\in \left\{\phi'\right\}^{\perp}$.
	
	On the other hand, a Taylor expansion of $V$ around $\phi$  reveals that
	\begin{equation}\label{eq002}
	V(u)=V(\phi)+ \langle V'(\phi),u-\phi\rangle+\frac{1}{2} \langle V''(\phi)(u-\phi),u-\phi\rangle +h(u),
	\end{equation}
	where $\lim\limits_{u\to\phi}\frac{h(u)}{||u-\phi||_X^2}=0$.
	Thus, we can choose $\alpha_1>0$ such that
	\begin{equation}\label{limit1}|h(u)|\leq\frac{\tau}{4}||u-\phi||_X^2, \qquad \mbox{for all}  \ u\in B_{\alpha_1}(\phi),\end{equation}
	where $B_{\alpha_1}(\phi)=\left\{u\in X; ||u-\phi||_X <\alpha_1 \right\}$.
	
	Since $V(\phi)=0$ and $V'(\phi)=0$, we have from (\ref{eq001}), (\ref{eq002}) and $(\ref{limit1})$ that
	\begin{equation}\label{eq003}
	V(u)\geq \frac{\tau}{4}\rho(u,\phi)^2,
	\end{equation}
	for all $u\in B_{\alpha_1}(\phi)$ such that $(u-\phi)\in \left\{\phi'\right\}^{\perp}.$
	
	Now, let us define the smooth map $S:X\times \mathbb{R}\rightarrow\mathbb{R}$ given by $S(u,r)=\langle u(\cdot-r),\phi'\rangle$. Since $S(\phi,0)=0$ and 
	$\frac{\partial S}{\partial r}(\phi,0)=-\langle\phi',\phi'\rangle\neq0$, we guarantee, from the implicit function theorem, the existence of $\alpha_2>0$, $\delta_0>0$ and a unique $C^1-$map $r:B_{\alpha_2}(\phi)\rightarrow(-\delta_0,\delta_0)$ such that $r(\phi)=0$ and $S(u,r(u))=0$, for all $u\in B_{\alpha_2}(\phi)$. Consequently, $(u(\cdot-r(u))-\phi)\in \left\{\phi'\right\}^{\perp}$, for all $u\in B_{\alpha_2}(\phi)$.
	
	To complete the proof, let $u\in U_{\alpha}$ with $\alpha>0$ arbitrarily fixed. Thus, there exists $r_1\in\mathbb{R}$ such that $\|u_1-\phi\|_X < \alpha$, where $u_1:=u(\cdot-r_1)$. Hence, 
	\begin{equation} \label{eq004}
	(u_1(\cdot-r(u_1))-\phi)\in \left\{\phi'\right\}^{\perp} \quad\mbox{if}\quad \alpha<\alpha_2.
	\end{equation}
	
	On other hand, using the fact that $r$ is continuous and $r(\phi)=0$, one has that there exists $\alpha_3>0$ such that 
	\begin{equation}\label{eq005}
	||u_1(\cdot-r(u_1))-u_1||_X<\frac{\alpha_1}{2}\quad\mbox{if}\quad \alpha<\alpha_3.
	\end{equation}

	 Let us consider $\alpha=\min\left\{\alpha_1/2,\alpha_2,\alpha_3\right\}$. Therefore, we conclude, by $(\ref{eq004})$ and $(\ref{eq005})$, 
	 \begin{eqnarray}
	 \|u_1(\cdot-r(u_1))-\phi\|_X&\leq&\|u_1(\cdot-r(u_1))-u_1\|_X+\|u_1-\phi\|_X \nonumber \\
	 &<&\frac{\alpha_1}{2}+\frac{\alpha_1}{2}=\alpha_1 \nonumber
	 \end{eqnarray}
	 and $(u_1(\cdot-r(u_1))-\phi)\in \left\{\phi'\right\}^{\perp}$. Since $V(u)=V(u_1(\cdot-r_1))$, we obtain, by $(\ref{eq003})$, the existence of $D>0$ such that $V(u)\geq D\rho(u,\phi)^2$.  
\end{proof}

The above lemma is the key point to prove our main result. Roughly speaking, it says that $V$ is a suitable Lyapunov function to handle with our problem. Finally, we present the proof our stability result.

\begin{proof}[Proof of Theorem \ref{teoest}]
	Let $\alpha>0$ be the constant such that Lemma \ref{lemma2} holds.  Since $V$ is continuous at $\phi$, for a given $\varepsilon>0$, there exists $\delta\in (0,\alpha)$ such that if $||u_0-\phi||_X<\delta$  one has
	\begin{equation*}
	V(u_0)=V(u_0)-V(\phi)<D\varepsilon^2,
	\end{equation*}
	where $D>0$ is the constant in Lemma \ref{lemma2}.
	
	The continuity in time of the function $\rho(u(t),\phi)$ allows to choose $T>0$ such that \be\label{subalpha}\rho(u(t),\phi)<\alpha,\ \ \  \mbox{for all}\ t\in [0,T).\ee
	Thus, one obtains $u(t)\in U_{\alpha}$, for all $t\in[0,T)$. Combining Lemma \ref{lemma2} and the fact that  $V(u(t))=V(u_0)$ for all $t\geq0$, we have
	\be\label{estepsilon1}
	\rho(u(t),\phi)<\varepsilon,\ \ \ \ \ \mbox{for all}\ t\in[0,T).
	\ee
	Next, we prove that $\rho(u(t),\phi)<\alpha$, for all $t\in [0,+\infty)$, from which one concludes the orbital stability. Indeed, let  $T_1>0$ be the supremum of the values of $T>0$ for which $(\ref{subalpha})$ holds. To obtain a contradiction, suppose that $T_1<+\infty$.  By choosing $\varepsilon<\frac{\alpha}{2}$ we obtain, from $(\ref{estepsilon1})$,
	$$
	\rho(u(t),\phi)<\frac{\alpha}{2}, \ \ \ \ \ \mbox{for all}\ t\in[0,T_1).
	$$
	Since $t\in(0,+\infty)\mapsto\rho(u(t),\phi)$ is continuous, there is $T_0>0$ such that
	$\rho(u(t),\phi)<\frac{3}{4}\alpha<\alpha$, for $t\in [0,T_1+T_0)$, contradicting the maximality of $T_1$. Therefore, $T_1=+\infty$ and the theorem  is established.
\end{proof}

\section{Sufficient condition for orbital stability}\label{suffstab}

In this section we will give a sufficient condition for the existence of the element $\Phi$ as assumed in Theorem \ref{teoest}. In particular, the main result of the section states that under assumption $(H)$, the periodic wave $\phi$ is orbitally stable provided that the quantity $s(\phi)$, defined in Corollary \ref{coro6789}, is positive. We point out that such a quantity does not depend on any derivative with respect to parameters.

\subsection{Regularity}

Let us start by proving that any solution of \eqref{ode-wave} is in fact smooth. This result will be used below and is the content of the next statement.

\begin{proposition}\label{regularity}
Let $m_1>1/3$. If $\psi\in X:=H_{per}^\frac{m_1}{2}([0,L_0])$ is a solution of $(\ref{ode-wave})$ in the sense of distributions, then $\psi \in H_{per}^n([0,L_0])$, for all $n\in\mathbb{N}$.
\end{proposition}
\begin{proof}
	In view of the embedding $H_{per}^{s_2}([0,L_0])\hookrightarrow H_{per}^{s_1}([0,L_0])$, $s_2\geq s_1>0$, it suffices to assume $1/3<m_1<1/2$.
 First, we will prove that $\psi\in L^{\infty}_{per}([0,L_0])$. Indeed, applying the Fourier transform in (\ref{ode-wave}) yields
\begin{equation*}
\widehat{\psi}(k)=\frac{\widehat{g(\psi)}(k)}{\omega\theta(k)+(\omega-1)}, \quad k\in\mathbb{Z},
\end{equation*}
where $g(\psi)=\frac{\psi^2}{2}-A$. Since $\psi\in X$, it follows that $\psi\in L^p_{per}([0,L_0])$ and $g(\psi)\in L^{\frac{p}{2}}_{per}([0,L_0])$, for all $2\leq p\leq 2/(1-m_1)$. Hence, by Hausdorff-Young inequality, we have
$\widehat{g(\psi)}\in \ell^q$ for all $1/m_1\leq q \leq \infty$. 

On other hand, by (\ref{A1A2}), $\left(\omega\theta(k)+(\omega-1)\right)^{-1}\in \ell^p$ for all $p>1/m_1.$ Let $\varepsilon>0$ be a small number such that $1\leq2/(1+m_1+\varepsilon)$. Thus
\begin{eqnarray*}
\|\widehat{\psi}\|^{2/(1+m_1+\varepsilon)}_{\ell^{2/(1+m_1+\varepsilon)}}&=&\|\widehat{\psi}^{2/(1+m_1+\varepsilon)}\|_{\ell^1} \nonumber \\
&\leq&\|\widehat{g(\psi)}^{2/(1+m_1+\varepsilon)}\|_{\ell^q}\|\left(\omega\theta(k)+(\omega-1)\right)^{-2/(1+m_1+\varepsilon)}\|_{\ell^{q'}},
\end{eqnarray*}
where $q,q'>0$ and $1/q+1/q'=1$. Now, we consider the smallest $q$ such that the first term on the right side is finite. That is, $q=(1+m_1+\varepsilon)/2m_1$. Thus $q'=(1+m_1+\varepsilon)/(1-m_1+\varepsilon)$. In order to obtain that the second term on the right side is finite we need the condition $1/m_1<2q'/(1+m_1+\varepsilon)$ which gives the inequality $1+\varepsilon<3m_1$. Note that $\varepsilon>0$ can always be chosen such that this holds since $m_1>1/3$. Therefore, we get $\widehat{\psi}\in\ell^{2/(1+m_1+\varepsilon)}$ which implies that there exists $\xi\in L^{2/(1-m_1-\varepsilon)}_{per}([0,L_0])$ such that $\widehat{\xi}=\widehat{\psi}$ (see \cite[page 190]{Zygmund}). Hence, using \cite[Corollary 1.51]{Zygmund} we obtain $\xi=\psi$ and so $g(\psi)\in L^{p}_{per}([0,L_0])$ for $1\leq p\leq 1/(1-m_1-\varepsilon)$ and $\widehat{g(\psi)}\in L^{p}_{per}([0,L_0])$ for $1/(m_1+\varepsilon)\leq p \leq \infty$. By iterating the procedure a finite number of times, we obtain 
\begin{equation}\label{l1fourier}
\widehat{\psi}\in\ell^1
\end{equation}
and thus $\psi\in L^{\infty}_{per}([0,L_0])$.

Finally, Plancherel's theorem leads to
\begin{eqnarray}
\|\mathcal{M}\psi\|_{L^2_{per}}&=&\left\|(\omega \mathcal{M}+(\omega-1))^{-1}\mathcal{M}g(\psi)\right\|_{L^2_{per}}=\left\|\frac{\theta(k)}{\omega \theta(k)+(\omega-1)}\widehat{g(\psi)}\right\|_{\ell^2} \nonumber \\
&\leq&\|\widehat{g(\psi)}\|_{\ell^2} = \|g(\psi)\|_{L^2_{per}} \leq \|\psi\|_{L^{\infty}_{per}}\|\psi\|_{L^2_{per}}+A\sqrt{L_0}<\infty, \nonumber
\end{eqnarray}
which implies $\psi \in H_{per}^{m_1}([0,L_0])$. Furthermore, from (\ref{l1fourier}), we have
\begin{eqnarray}
\|\mathcal{M}^2\psi\|_{L^2_{per}}&=&\left\|(\omega \mathcal{M}+(\omega-1))^{-1}\mathcal{M}^2g(\psi)\right\|_{L^2_{per}}=\left\|\frac{\theta(k)^2}{\omega \theta(k)+(\omega-1)}\widehat{g(\psi)}\right\|_{\ell^2} \nonumber \\
&\leq&\|\theta(k)\widehat{g(\psi)}\|_{\ell^2} = \|\mathcal{M}g(\psi)\|_{L^2_{per}} \nonumber \\
&\leq&\|g(\psi)\|_{H^{m_1}_{per}} \leq \|\psi^2\|_{H^{m_1}_{per}}+A\sqrt{L_0} \nonumber \\
&=&\|(1+|k|^2)^{\frac{m_1}{2}}(\widehat{\psi}\ast\widehat{\psi})(k)\|_{\ell^2}+ A\sqrt{L_0} \nonumber \\
&\leq& K_{m_1}\left[\|\widehat{\psi}\|_{\ell^1}\|\widehat{\psi}\|_{\ell^2}+2\|(\cdot)^{m_1}\widehat{\psi}\|_{\ell^2}\|\widehat{\psi}\|_{\ell^1}\right]+A\sqrt{L_0}<\infty, \nonumber
\end{eqnarray}
where $K_{m_1}>0$ is a constant depending only on $m_1$. After iterations, we conclude that $\psi \in H_{per}^n([0,L_0])$, for all $n\in\mathbb{N}$.
\end{proof}

\subsection{Existence of a smooth surface of periodic waves} As an intermediate step to obtain our main result, we will prove that $(H)$ is sufficient to show the existence of a smooth surface of periodic waves. This will be a consequence of the Implicit Function Theorem.

\begin{theorem}\label{existcurve}
Suppose that assumption (H) holds. Then, there exist an open neighborhood $\mathcal{O}$ containing $(\omega_0,A_0)$ and a smooth surface 
\begin{equation*}
(\omega,A)\in \mathcal{O}\mapsto \phi_{(\omega,A)} \in H_{per,e}^n([0,L_0]), \quad n\in\mathbb{N},
\end{equation*}
of even $L_0$-periodic solutions of (\ref{ode-wave}).  

In particular, $\phi_{(\omega,A)}\to \phi_{(\omega_0,A_0)}$, as $(\omega,A)\to(\omega_0,A_0)$, in $L^\infty_{per}([0,L_0])$.
\end{theorem}
\begin{proof}
We first define $\Upsilon:(0,+\infty)\times\R\times H_{per,e}^{m_1}([0,L_0])\to L_{per,e}^2([0,L_0])$ by
\begin{equation}\label{Ups}
\Upsilon(\omega,A,f)=\omega\mathcal{M}f+(\omega-1)f-\frac{1}{2}f^2+A,
\end{equation}
where $H_{per,e}^{m_1}([0,L_0])$  indicates the periodic Sobolev space $H_{per}^{m_1}([0,L_0])$ constituted by even $L_0$-periodic functions. The case $\Upsilon:(-\infty,0)\times\R\times H_{per,e}^{m_1}([0,L_0])\to L_{per,e}^2([0,L_0])$ can be determine using similar arguments. Since $\theta$ is an even function we have that $\mathcal{M}f$ is also even, for any  function $f$ in $H_{per,e}^{m_1}([0,L_0])$. In addition, since $m_1>1/3$ the Sobolev embedding implies $f^2\in L_{per}^{2}([0,L_0])$. This means that \eqref{Ups} makes sense in $ L_{per,e}^2([0,L_0])$.

 From assumption \textit{(H)} one has $\Upsilon(\omega_0,A_0,\phi)=0$. Moreover, note that $\Upsilon$ is smooth and its Fr\'echet derivative with respect to $\phi$ evaluated at $(\omega_0,A_0,\phi)$ is 
$$\mathcal{G}:=\omega_0\mathcal{M}+(\omega_0-1)-\phi.$$ 
From \eqref{ode-wave} it is easily seem that $\phi'$ is an eigenfunction of the operator $\mathcal{G}$ (defined on $L_{per}^2([0,L_0])$ with domain $H_{per}^{m_1}([0,L_0])$) whose eigenvalue is $\lambda=0.$
Since $\phi'$ does not belong to $H_{per,e}^{m_1}([0,L_0])$ (because it is odd), we conclude that $\mathcal{G}$ is one-to-one. 

Next, let us prove that $\mathcal{G}$ is also surjective. Indeed, $\mathcal{G}$ is clearly a self-adjoint operator. Thus, the spectrum of $\mathcal{G}$, denoted by $\sigma(\mathcal{G})$ is such that $\sigma(\mathcal{G})=\sigma_{disc}(\mathcal{G})\cup\sigma_{ess}(\mathcal{G})$, where $\sigma_{disc}(\mathcal{G})$ and $\sigma_{ess}(\mathcal{G})$ stand, respectively, for the discrete and essential spectra. Being $H_{per,e}^{m_1}([0,L_0])$ compactly embedded in $L_{per,e}^2([0,L_0])$, the operator $\mathcal{G}$ has compact resolvent. Consequently, $\sigma_{ess}(\mathcal{G})=\emptyset$ and $\sigma(\mathcal{G})=\sigma_{disc}(\mathcal{G})$ consists of isolated eigenvalues with finite algebraic multiplicities (see also Proposition 3.1 in \cite{AN}). Finally, since $\mathcal{G}$ is one-to-one, it follows that $0$ is not an eigenvalue of $\mathcal{G}$, and so it does not belong to $\sigma(\mathcal{G})$. Therefore, $0\in \rho(\mathcal{G})$, where $\rho(\mathcal{G})$ denotes the resolvent set of $\mathcal{G}$, and consequently, by definition, $\mathcal{G}$ is surjective. 

The arguments above imply that $\mathcal{G}^{-1}$ exists and is a bounded linear operator. Thus, since $\Upsilon$ and its derivative with respect to $f$  are smooth maps on their domains, from the Implicit Function Theorem (see, for instance, Theorem 15.1 in \cite{Deimling}) and Proposition \ref{regularity} we establish the desired results.
\end{proof}

Next result shows that the spectral property in $(H)$ is preserved by small perturbations of the parameter $(\omega,A)$ in an open subset containing $(\omega_0,A_0)$.

\begin{proposition}\label{prop1}
Suppose that assumption (H) holds and let $\phi_{(\omega,A)}$ be the periodic traveling wave solution obtained in Theorem \ref{existcurve}. Then, for all $(\omega,A)\in\mathcal{O}$, operator $\mathcal{L}_{(\omega,A)}=\omega\mathcal{M}+(\omega-1)-\phi_{(\omega,A)}$ has only one negative eigenvalue which is simple and zero is a simple eigenvalue whose eigenfunction is $\phi_{(\omega,A)}'$
\end{proposition}
\begin{proof}
Assume $(\omega,A)\in \mathcal{O}$ and define  
$$
\widetilde{\mathcal{L}}_{(\omega,A)}:=\frac{1}{\omega}\mathcal{L}_{(\omega,A)}=\mathcal{M}+\frac{\omega-1}{\omega}-\frac{\phi_{(\omega,A)}}{\omega}.
$$ 
It is clear that such an operator defined on $L_{per}^2([0,L_0])$ with domain $D(\mathcal{L})=H_{per}^{m_1}([0,L_0])$ is also self-adjoint. Thus, since $\mathcal{L}_{(\omega,A)}=\omega \widetilde{\mathcal{L}}_{(\omega,A)}$ and $\omega\neq0$, it suffices to prove that the statements in the  proposition holds for $\widetilde{\mathcal{L}}_{(\omega,A)}$.

Let us first show that $\widetilde{\mathcal{L}}_{(\omega,A)}$ converges to $\widetilde{\mathcal{L}}_{(\omega_0,A_0)}$, as $(\omega,A)\to(\omega_0,A_0)$, in the metric  \textit{gap} $\widehat{\delta}$ (see Sections 2 and 3 of Chapter IV in \cite{kato1}). 
Indeed, since the multiplication operator $\Lambda:L_{per}^{2}([0,L_0])\to L_{per}^{2}([0,L_0])$ defined by $\Lambda f=(\phi_{(\omega,A)}/\omega)f$ is bounded with norm $\|\Lambda\|\leq \|\phi_{(\omega,A)}\|_{L^\infty_{per}}/\omega$, Theorem 2.17 in \cite[Chapter IV]{kato1} implies that
\begin{equation}\label{deltati1}
\begin{split}
\widehat{\delta}(\widetilde{\mathcal{L}}_{(\omega_0,A_0)},\widetilde{\mathcal{L}}_{(\omega,A)}) \leq &\, 2\left(1+\frac{\|\phi_{(\omega,A)}\|_{L^{\infty}_{per}}^2}{\omega^2}\right)\widehat{\delta}\left(\widetilde{\mathcal{L}}_{(\omega_0,A_0)}+\frac{\phi_{(\omega,A)}}{\omega},\mathcal{M}+\frac{\omega-1}{\omega}\right) 
\end{split}
\end{equation}
Now, by using the multiplication operator
$$
f\mapsto \left[\frac{\omega_0-1}{\omega_0}-\frac{\omega-1}{\omega}+\frac{\phi_{(\omega,A)}}{\omega}-\frac{\phi_{(\omega_0,A_0)}}{\omega_0}\right]f
$$
is also bounded in $L_{per}^{2}([0,L_0])$ with norm below 
$$
\left|\frac{\omega_0-1}{\omega_0}-\frac{\omega-1}{\omega}\right|+\left\|\frac{\phi_{(\omega,A)}}{\omega}-\frac{\phi_{(\omega_0,A_0)}}{\omega_0}\right\|_{L^{\infty}_{per}},
$$
an application of Theorem 2.14  in \cite[Chapter IV]{kato1} yields
\begin{equation}\label{deltati}
\begin{split}
 \widehat{\delta}\left(\widetilde{\mathcal{L}}_{(\omega_0,A_0)}+\frac{\phi_{(\omega,A)}}{\omega},\mathcal{M}+\frac{\omega-1}{\omega}\right) 
\leq &\left[\left|\frac{\omega_0-1}{\omega_0}-\frac{\omega-1}{\omega}\right|+\left\|\frac{\phi_{(\omega,A)}}{\omega}-\frac{\phi_{(\omega_0,A_0)}}{\omega_0}\right\|_{L^{\infty}_{per}}\right].
\end{split}
\end{equation}
By recalling that  $\phi_{(\omega,A)}\to \phi_{(\omega_0,A_0)}$, as $(\omega,A)\to(\omega_0,A_0)$, in $L^\infty_{per}([0,L_0])$, a combination of \eqref{deltati1} and \eqref{deltati} finally establish that $\widehat{\delta}(\widetilde{\mathcal{L}}_{(\omega_0,A_0)},\widetilde{\mathcal{L}}_{(\omega,A)})\to 0$, as $(\omega,A)\to(\omega_0,A_0)$. 

Consequently, by taking into account that zero is an eigenvalue of $\widetilde{\mathcal{L}}_{(\omega,A)}$ with eigenfunction $\phi_{(\omega,A)}'$, from Theorem 3.16  in \cite[Chapter IV]{kato1}, we conclude that  for $(\omega,A)$ in a neighborhood of $(\omega_0,A_0)$, $\widetilde{\mathcal{L}}_{(\omega,A)}$ has the same spectral properties of $\mathcal{L}_{(\omega_0,A_0)}$, which is to say that it has only one negative eigenvalue which is simple and zero is a simple eigenvalue. At this point, it should be clear that if necessary we can take a neighborhood smaller than $\mathcal{O}$. However, for convenience we assume that such a set is the whole $\mathcal{O}$. 
\end{proof}

 Since we have obtained a smooth surface of periodic solutions with a fixed period $L_0>0$, we can define
$$
\eta:=\frac{\partial}{\partial\omega}\phi_{(\omega,A)}\Big|_{(\omega_0,A_0)},\ \qquad\beta:=\frac{\partial}{\partial A}\phi_{(\omega,A)}\Big|_{(\omega_0,A_0)},
$$
Next we set
 $$
 M_{\omega}(\phi)=\frac{\partial}{\partial\omega}\int_0^{L_0}\phi_{(\omega,A)}(x)dx\Big|_{(\omega_0,A_0)},\qquad  M_{A}(\phi)=\frac{\partial}{\partial A}\int_0^{L_0}\phi_{(\omega,A)}(x)dx\Big|_{(\omega_0,A_0)},
 $$

 $$
 F_{\omega}(\phi)=\frac{1}{2}\frac{\partial}{\partial{\omega}}\int_0^{L_0}\Big(\phi_{(\omega,A)}\mathcal{M}\phi_{(\omega,A)}+\phi_{(\omega,A)}^2(x)\Big)dx\Big|_{(\omega_0,A_0)},$$
and  $$F_{A}(\phi)=\frac{1}{2}\frac{\partial}{\partial{A}}\int_0^{L_0}\Big(\phi_{(\omega,A)}\mathcal{M}\phi_{(\omega,A)}+\phi_{(\omega,A)}^2(x)\Big)dx\Big|_{(\omega_0,A_0)}.
 $$

These quantities will be very useful in what follows.

\begin{proposition}\label{propKpos}
	Let $\Delta:\R^2\to\R$ be the function defined as
	$$
	\Delta(x,y)=x^2F_\omega(\phi)+xy(M_\omega(\phi)+F_A(\phi))+y^2M_A(\phi).
	$$
	Assume the existence of  $(x_0,y_0)\in\R^2$ such that $\Delta(x_0,y_0)>0$. Then, there exists $\Phi\in X$ such that   $\langle\mathcal{L}_0\Phi,\varphi\rangle=0$, for all $\varphi\in \Upsilon_0$, and
	\begin{equation*}
	\mathcal{I}=\langle \mathcal{L}_0\Phi,\Phi \rangle<0.
	\end{equation*}
\end{proposition}
\begin{proof}
	It suffices to define  $\Phi:=x_0\eta+y_0\beta$. Indeed, since $\mathcal{L}_0\beta=-1$ and $\mathcal{L}_0\eta=-(\mathcal{M}\phi+\phi)$, it is clear that  $\langle\mathcal{L}_0\Phi,\varphi\rangle=0$, for all $\varphi\in \Upsilon_0$, and
	\[
	\begin{split}
	\langle\mathcal{L}_0\Phi,\Phi \rangle&=\langle-x_0\phi-y_0,x_0\eta+y_0\beta\rangle\\
	&=-(x_0^2F_\omega(\phi)+x_0y_0M_\omega(\phi)+x_0y_0F_A(\phi)+y_0^2M_A(\phi))\\
	&=-\Delta(x_0,y_0).
	\end{split}
	\]
	The proof is thus completed.
\end{proof}

%\begin{obs}
%A straightforward calculation enables us to deduce that if there exists $(x_0,y_0)\in\mathbb{R}^2$
%such that $\Delta(x_0,y_0)>0$, then at least one of the quantities $F_{\omega}(\phi)$, $M_{A}(\phi)$ or $M_{\omega}(\phi)^2-M_{A}(\phi)F_{\omega}(\phi)$ are positive. Thus, we can conclude that our assumption is, in some sense, weaker than the ones appearing in the current literature (see \cite{ASB}, \cite{hur}, \cite{johnsonBBM}, and references therein). 
%\end{obs}

Next result gives a sufficient condition to obtain $(x_0,y_0)$ satisfying $\Delta(x_0,y_0)>0$. Consequently, we are in conditions to prove Theorem \ref{coro14}.

\begin{corollary}\label{coro6789}
Suppose that assumption (H) holds. If $\omega_0-1-2A_0\neq0$ and 
$s(\phi)$ defined in $(\ref{sphi123})$ is positive,
there exists $(x_0,y_0)\in\R^2$ such that $\Delta(x_0,y_0)>0$. 
\end{corollary}
\begin{proof}
First, in order to simplify the notation we define $\psi=\phi_{(\omega,A)}$ as the solution obtained in Theorem $\ref{existcurve}$. Deriving equation (\ref{ode-wave}) with respect to $\omega$ and $A$, we obtain, respectively
\begin{equation}\label{eq1}
\mathcal{M}\psi+\omega \mathcal{M}\eta+\psi+(\omega-1)\eta-\psi\eta=0
\end{equation}
and
\begin{equation}\label{eq2}
\omega \mathcal{M}\beta+(\omega-1)\beta-\psi\beta+1=0.
\end{equation}
Next, if we integrate equations (\ref{eq1}) and (\ref{eq2}) over $[0,L_0]$ one has
\begin{equation}\label{eq3}
\frac{1}{2}\frac{\partial}{\partial \omega}\int_{0}^{L_0}\psi^2dx=M(\psi)+(\omega-1)M_{\omega}(\psi)
\end{equation}
and
\begin{equation}\label{eq4}
\frac{1}{2}\frac{\partial}{\partial A}\int_{0}^{L_0}\psi^2dx=L_0+(\omega-1)M_{A}(\psi),
\end{equation}
where we used, in view of \eqref{A1A2}, that
$$
\int_{0}^{L_0}\mathcal{M}fdx=\widehat{\mathcal{M}f}(0)=\theta(0)\widehat{f}(0)=0.
$$

On the other hand, multiplying (\ref{eq1}) by $\psi$ and (\ref{ode-wave}) by $\eta$, adding the results and using (\ref{eq3}), we get
\begin{equation}\label{eq5}
M(\psi)+\left(\omega-1-\frac{A}{2}\right)M_{\omega}(\psi)=F(\psi)+\omega F_{\omega}(\psi)-\frac{1}{4} \frac{\partial}{\partial \omega}\int_{0}^{L_0}\psi^3dx.
\end{equation}
Similarly, multiplying (\ref{eq2}) by $\psi$ and (\ref{ode-wave}) by $\beta$, adding the results and using (\ref{eq4}), we conclude
\begin{equation}\label{eq6}
L_0+\left(\omega-1-\frac{A}{2}\right)M_{A}(\psi)=\frac{1}{2}M(\psi) + \omega F_A(\psi)-\frac{1}{4} \frac{\partial}{\partial A}\int_{0}^{L_0}\psi^3dx.
\end{equation}

Now, multiplying (\ref{eq1}) by $\psi$, integrating over $[0,L_0]$ and using (\ref{ode-wave}) one has
\begin{equation}\label{eq7}
2F(\psi)-\frac{1}{6}\frac{\partial}{\partial \omega}\int_{0}^{L_0}\psi^3dx - AM_{\omega}(\psi)=0.
\end{equation}
Similarly, multiplying (\ref{eq2}) by $\psi$, integrating over $[0,L_0]$ and using (\ref{ode-wave}), we get
\begin{equation}\label{eq8}
M(\psi)-\frac{1}{6}\frac{\partial}{\partial A}\int_{0}^{L_0}\psi^3dx - AM_{A}(\psi)=0.
\end{equation}

Thus, deriving (\ref{eq7}) with respect to $A$,  (\ref{eq8}) with respect to $\omega$ and adding the results, we obtain the equality
\begin{equation}\label{eq9}
F_A(\psi)=M_{\omega}(\psi).
\end{equation}
So, comparing the results in (\ref{eq7}), (\ref{eq8}) and (\ref{eq9}) with (\ref{eq5}) and (\ref{eq6}), we conclude that
\begin{equation}\label{eq10}
M(\psi)+2F(\psi)+\left(\omega-1-2A\right)M_{\omega}(\psi)=\omega F_{\omega}(\psi).
\end{equation}
and
\begin{equation}\label{eq11}
L_0+\left(\omega-1-2A\right)M_{A}(\psi)+M(\psi)=\omega M_{\omega}(\psi).
\end{equation}

Finally, collecting the results in (\ref{eq10}) and (\ref{eq11}), considering $\omega_0-1-2A_0\neq 0$ and evaluating the results at  $(\omega_0,A_0)$, we have 

\[
\begin{split}
\Delta(x_0,y_0)=&\left(\frac{x_0^2(\omega_0-1-2A_0)}{\omega_0}+2x_0y_0+\frac{y_0^2\omega_0}{\omega_0-1-2A_0}\right)M_{\omega}(\phi) \\
&+ \frac{x_0^2}{\omega_0}\left(M(\phi)+2F(\phi)\right)-\frac{y_0^2}{\omega_0-1-2A_0}\left(M(\phi)+L_0\right).
\end{split}
\]

By choosing $y_0\neq 0$, $x_0=\frac{-\omega_0 y_0}{\omega_0-1-2A_0}$ and using the fact
\begin{equation*}
(\omega_0-1)M(\phi)+\frac{1}{2}\int_{0}^{L_0}\phi \mathcal{M}\phi dx+A_0L_0=F(\phi),
\end{equation*}
we get
\[
\begin{split}
\Delta(x_0,y_0)&=\frac{y_0^2}{(\omega_0-1-2A_0)^2}\Big[({2\omega_0(\omega_0-1)+2A_0+1})M(\phi)+\omega_0\int_{0}^{L_0}\phi \mathcal{M}\phi dx   \\
&\quad\quad +({2A_0(\omega_0+1)-\omega_0+1})L_0\Big]\\
&=\frac{y_0^2}{(\omega_0-1-2A_0)^2}s(\phi).
\end{split}
\] 
 The proof is thus completed.
\end{proof}

%%\begin{corollary}\label{coro13}
%%%Consider $\omega_0-1-2A_0\neq0$ and let $%%%(x_0,y_0)\in\mathbb{R}^2$ be the pair in Corollary %%$\ref{coro6789}$. Let us denote $\Gamma_1:=2\omega_0(\omega_0-1)+2A_0+1$ and $\Gamma_2:=x_0F(\phi)+y_0M(\phi)$.  If $\Gamma_1\Gamma_2\neq0$,  one has that $\Delta(x_0,y_0)>0$. 
%%%%%\end{corollary}
%%%%\begin{proof}
%%%%In fact, let us consider $\Gamma_3:=\omega_0\int_0^{L_0}\phi\mathcal{M}\phi dx$ and $\Gamma_4:=2A_0(\omega_0+1)-\omega_0+1$. Thus, $s(\phi)$ can be rewritten as $s(\phi)=\Gamma_1 M(\phi)+\Gamma_3+\Gamma_4$. From Corollary $\ref{coro6789}$ and the definition of $\Gamma_2$, we get 
%%%%%$$s(\phi)=\frac{\Gamma_1\Gamma_2}{y_0}+\Gamma_5\Gamma_1+\Gamma_3+\Gamma_4,$$
%%%%%where $\Gamma_5=\frac{\omega_0}{\omega_0-1-2A_0}F(\phi)$. If $\Gamma_1\Gamma_2\neq0$, we can choose $y_0$ such that $s(\phi)>0$. The result is established.
%%%\end{proof}

\section{Applications}\label{applic}

In this section, we apply the arguments developed in Section \ref{OSPW} in order to obtain the orbital stability of periodic waves for some regularized dispersive models.

\subsection{Orbital stability for a fifth-order model.} Here, as an application of Corollary \ref{coro14}, we present the orbital stability of a periodic traveling-wave solution related to the following fifth-order model
\begin{equation}\label{rkawa}
u_t+u_x+uu_x+u_{xxxxt}=0.
\end{equation}
 Equation $(\ref{rkawa})$ can be seen as the regularized version of 
\begin{equation*}
u_t+uu_x-u_{xxxxx}=0,
\end{equation*}
which models wave propagation on a nonlinear transmission line (see \cite{Kano}).

To simplify the exposition, throughout this subsection we assume $L_0=2\pi$.
 Note that \eqref{rkawa} is of the form $(\ref{rDE})$  with $\mathcal{M}=\partial_x^4$. In particular, $\theta(\kappa)=\kappa^4$ and the energy space is $X=H^2_{per}([0,2\pi])$.

By looking for periodic traveling wave solutions having the form $u(x,t)=\phi(x-\omega_0 t)$, we get from $(\ref{rkawa})$ (after integration) that $\phi=\phi_{(\omega_0,A_0)}$ solves the nonlinear ordinary differential equation
\begin{equation}\label{edokawa}\omega_0\phi''''+(\omega_0-1)\phi-\frac{1}{2}\phi^2+A_0=0.
\end{equation}

 Equation $(\ref{edokawa})$ admits an explicit $2\pi$-periodic solution given by the ansatz (see \cite{parkes})
\begin{equation}\label{solkawa}
\begin{split}
\phi(x)&=a + b\left[\dn^2\left(\frac{K(k)}{\pi}x,k\right)-\frac{E(k)}{K(k)}\right]\\
&\quad+ d\left[\dn^4\left(\frac{K(k)}{\pi}x,k\right)-(2-k^2)\frac{2E(k)}{3K(k)}+\frac{1-k^2}{3}\right],\end{split}
\end{equation}
where
$$k_0=\frac{\sqrt{2}}{2}, \quad a=\frac{-28K(k_0)^4\omega_0+\pi^4(\omega_0-1)}{\pi^4}, \quad b=\frac{-1680\omega_0K(k_0)^4}{\pi^4}$$ 
and $$d=\frac{1680\omega_0K(k_0)^4}{\pi^4}$$
Also, $\dn$ represents the Jacobi elliptic function of  dnoidal type, $K=K(k)$ is the complete elliptic integral of the first kind, $E=E(k)$ is the complete elliptic integral of the second kind and both of them depend on the elliptic modulus $k\in(0,1)$ (see \cite{bird} for additional details). It is to be pointed out that $\omega_0$ is a free parameter and we shall assume that $\omega_0>0$ for the sake of completeness. Moreover, constant $A_0$ is a smooth function depending $\omega_0$ given by
\begin{equation}\label{valueA}
A_0=\frac{23184\omega_0^2K(k_0)^8-\pi^8(\omega_0-1)^2}{\pi^8}.
\end{equation}

Next, we will obtain the spectral properties related to the operator $\mathcal{L}_0=\omega_0\partial^4_x+(\omega_0-1)-\phi$ as required in $(H)$. To do so, we will  utilize the following result of \cite{AN}:

\begin{theorem}\label{ANtheorem}
	Suppose that $\phi$ is a positive even solution of (\ref{edokawa}) such that $\widehat{\phi}>0$ and $\frac{\partial^2}{{\partial x}^2}(\log{g(x)})<0$, $x\neq0$, where $g$ is a real function such that $g(n)=\widehat{\phi}(n)$, $n\geq0$. Then the operator $\mathcal{L}_0$ has only one negative eigenvalue which is simple and zero is a simple eigenvalue whose eigenfunction is $\phi'$.
\end{theorem}
\begin{proof}
See Theorem 4.1 and Lemma 4.1 in \cite{AN}. See also \cite{albert1} for the continuous case.
\end{proof}

The application of Theorem \ref{ANtheorem} is not immediate. First of all note that $\phi$ is not positive (see Figure \ref{figurad2logg} (Left)). The idea to overcome this is to use an  auxiliary function defined by $\varrho:=\mu+\phi$, where  $\mu\in\R$ is a fixed arbitrary number such that $\varrho>0$. Note that $\varrho$ is a solution of the equation
$$\omega_0\varrho''''+(\omega_0-1+\mu)\varrho-\frac{1}{2}\varrho^2+\widetilde{A_0}=0$$
where $\widetilde{A_0}=A_0-\mu(\omega_0-1)-\frac{\mu^2}{2}$. Moreover, we can rewrite $\mathcal{L}_0$ as
$$
\mathcal{L}_0=\omega_0\partial^4_x+(\omega_0-1)-\phi=\omega_0\partial^4_x+(\omega_0-1+\mu)-\varrho.
$$
Now, we are going to determine the nonpositive spectrum of $\mathcal{L}_0$ according with equality above

In fact, by \cite{ayse} the solution $\phi$ in (\ref{solkawa}) has the Fourier expansion
\begin{eqnarray}
\phi(x) &=& a+\sum_{n=1}^{\infty}\gamma(n) n\mbox{csch}\left(\frac{n\pi K(k_0')}{K(k_0)}\right)\cos\left(\frac{2\pi n}{L_0}x\right), \nonumber \\
&=&a+\sum_{n=1}^{\infty}\gamma n\mbox{csch}\left(n\pi\right)\cos\left(nx\right), \nonumber
\end{eqnarray}
where $\gamma=\gamma(n)$ is defined by
\begin{equation}
\gamma=\frac{b\pi^2}{K(k_0)^2}+\frac{d\pi^2}{k_0^2K(k_0)^2}\left(\frac{4-2k_0^2}{3}+\frac{n^2\pi^2}{6K(k_0)}\right) =560\omega_0K(k_0)n^2+\frac{1680\omega_0K(k_0)^2}{\pi^2} \nonumber
\end{equation}
and we have used that $k'_0=\sqrt{1-k_0^2}=k_0$.
Therefore, the Fourier coefficients of $\phi$ are given by
\begin{equation*}
\widehat{\phi}(n)=\left \{
\begin{array}{cc}
a, & n=0 \\
\frac{\gamma}{2} n\mbox{csch}\left(n\pi\right), & n\neq0. \\
\end{array}
\right.
\end{equation*}

By considering $g(x):=\frac{\gamma(x)}{2} x\mbox{csch}\left(x\pi\right)$, $x\in\mathbb{R}$ and choosing $\mu$ large enough such that $\varrho>0$ and $\hat{\varrho}(0)=a+\mu>g(0)$, it is possible we redefine function $g$
by a differentiable function $p:\mathbb{R}\rightarrow\mathbb{R}$
such that $p(0)= a+\mu$ and
$p(x)=g(x)$ in $(-\infty,-1]\cup[1,+\infty)$ with  $\frac{\partial^2}{{\partial x}^2}(\log{p(x)})<0$ for $x\neq0$ (see \cite[page 1145]{AN}).
So, using  Theorem \ref{ANtheorem}, we obtain that $\mathcal{L}_0$ has only one negative eigenvalue which is simple and zero is a simple eigenvalue whose eigenfunction is $\phi'$. Therefore, we have that assumptions in \textit{(H)} hold. 

\begin{figure}
	\begin{center}		
		\includegraphics[scale=0.15]{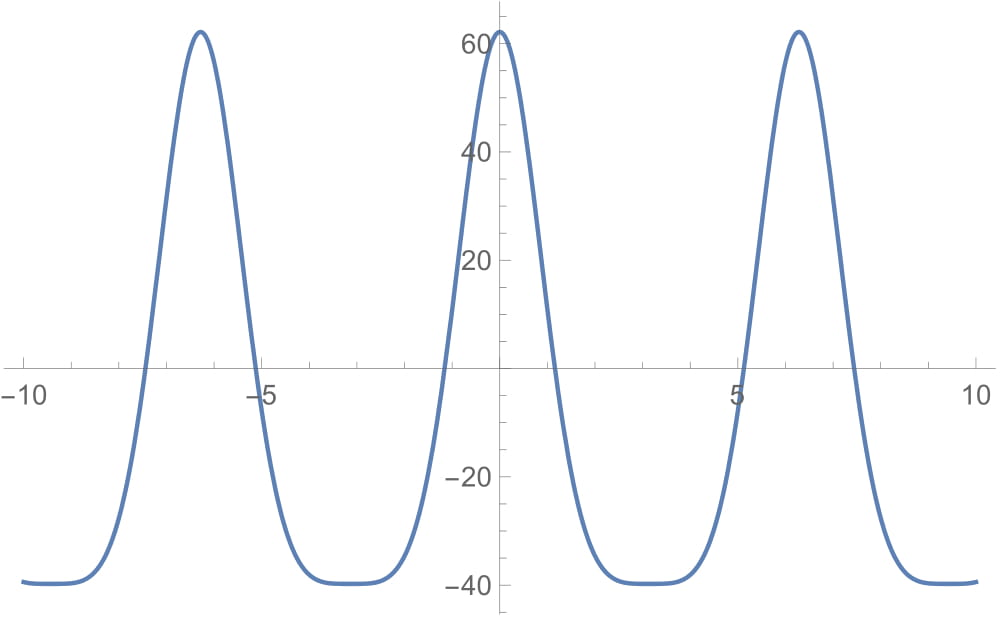} 
		\quad
		\includegraphics[scale=0.13]{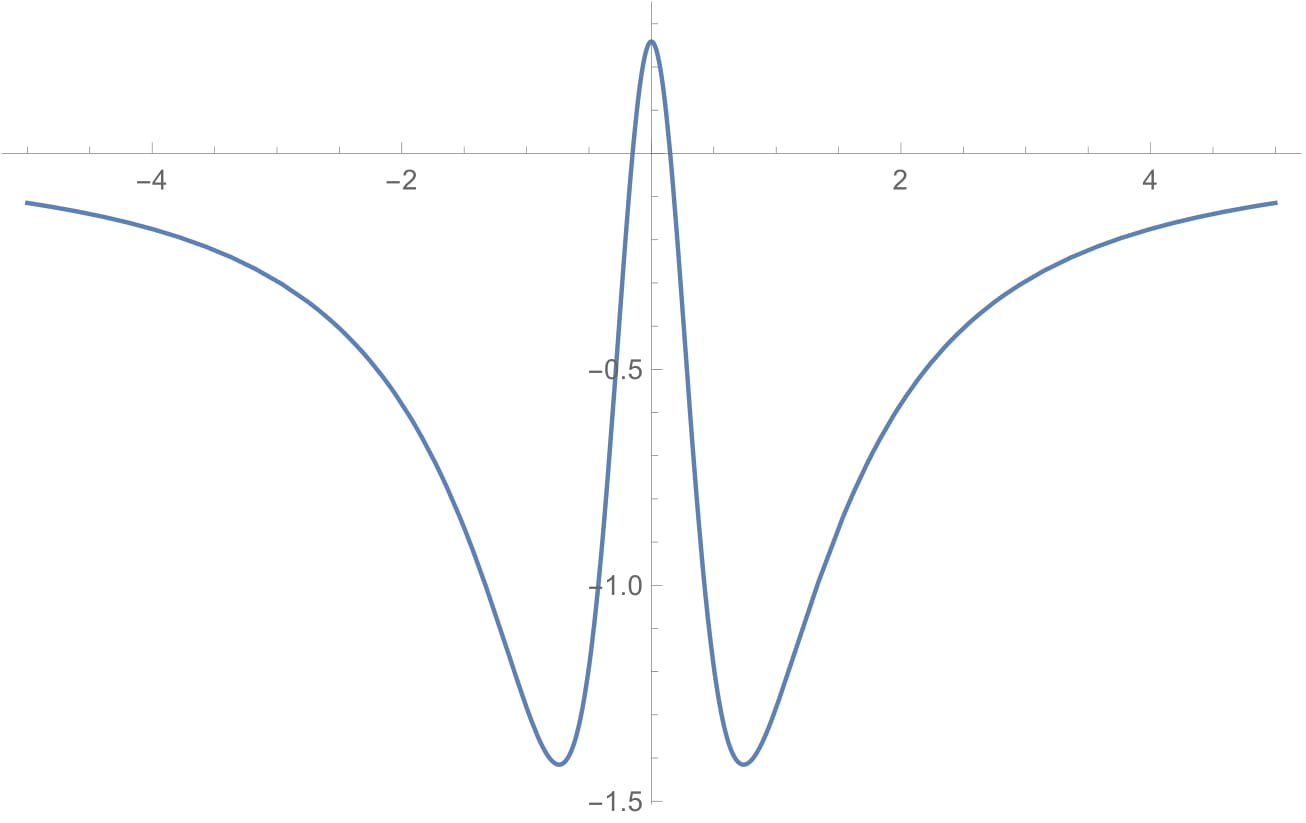}
		\caption{Left: Graph of $\phi$ for $\omega_0=2$. Right: Graph of $\frac{\partial^2}{\partial_{x}^2}\left(\log g(x)\right)$.}
		\label{figurad2logg}	
	\end{center}
\end{figure}

Finally, by using $(\ref{solkawa})$ we obtain after some straightforward but tedious calculations, that
\[
\begin{split}
s(\phi)&=\left(2\omega_0(\omega_0-1)+2A_0+1\right)aL_0+\omega_0\int_{0}^{L_0}(\phi'')^2dx+\left(2A_0(\omega_0+1)-\omega_0+1\right)L_0\\
&\approx 5277.03\omega_0^3, 
\end{split}
\]
from which we conclude that $s(\phi)>0$ for all $\omega_0>0$. It is also clear, from (\ref{valueA}), that
$\omega_0-1-2A_0\neq 0$ for all $\omega_0\in (0,c_0)\cup (c_0,+\infty)$, where $c_0\approx 0.0513569$ is the unique positive root of $\omega_0-1-2A_0$. Therefore, from Theorem \ref{coro14}, we conclude that $\phi$ is orbitally stable in $H_{per}^2([0,L_0])$ by the periodic flow of (\ref{rkawa}).

%Next, after some tedious calculations, we obtain from $(\ref{solkawa})$ that
%\begin{eqnarray}
%s(\phi)&=&\left(2\omega_0(\omega_0-1)+2A_0+1\right)aL_0+\omega_0\int_{0}^{L_0}(\phi'')^2dx+\left(2A_0(\omega_0+1)-\omega_0+1\right)L_0 \nonumber \\
%&=& \frac{-\omega_0^3K(k_0)^4}{L_0^{12}}j_1(L_0) \nonumber
%\end{eqnarray}
%where $$j_1(L_0)=448L_0^8+2658926592K(k_0)^8-11870208K(k_0)^4L_0^4-3.897406155\cdot10^9K(k_0)^7L_0.$$

%Using (\ref{valueA}), it is possible to write 
%$$\omega_0-1-2A_0=\frac{\omega_0\left(L_0^8(\omega_0-1)-5935104K(k_0)^8\omega_0\right)}{L_0}$$
%which it is different from zero for $\omega_0\neq L_0/\left(L_0-5935104K(k_0)^8\right)$.
%Now, let us analyze the $j_1$ signal. Indeed, first we observe that the function $j_1$ possesses two zeros, namely, $L_{0,1}=1.263686400$ and $L_{0,2}=24.42237205$. Moreover, it is easy to evaluate that $j_1'(L_{0,1})<0.$ Hence, we can affirm that $j_1(L_0)<0$ if $1.263686400<L_0<24.42237205$.

%We can also plot the graphic of $j_1$ in order to have a visual context of the value of $s(\phi)$% For instance, if $2<L_0<24$ and $\omega_0>1$, we deduce that $s(\phi)>0$ 
%(see Figure $(\ref{figure1})$). Therefore, from Theorem \ref{teoest}, we conclude that $\phi$ in (\ref{solkawa}) is orbitally stable in $H_{per}^2([0,L_0])$ by the periodic flow of the equation (\ref{rkawa}).

%\begin{figure}[!htb]
%	\centering
%	\includegraphics[width=5cm]{figure1.jpg}
%	\caption{Graphic of $j_1$}
%	\label{figure1}
%\end{figure}

\subsection{Minimizers and orbital stability of periodic waves}

In this subsection, we present a simple way to prove the orbital stability of periodic waves for equation $(\ref{rDE})$ provided they minimize a convenient smooth functional with a constraint. In other words, we show that, in this case, the hypothesis $(H)$ and the fact $s(\phi)>0$ can be replaced by the simple assumption that $\phi$ is even and $\ker(\mathcal{L}_{0})=[\phi']$.

Let $L_0>0$ be fixed. For $\gamma>0$ define the set
$$Y_{\gamma}=\left\{u\in X;\ \int_0^{L_0}u^3=\gamma\right\}.$$
Our first goal is to find a minimizer of the constrained minimization problem
\begin{equation}
\label{infB}
m=\inf_{ u\in Y_{\gamma}}B(u),
\end{equation}
where, for $\omega_0>1$ fixed,
\[
\begin{split}
B(u)&=\frac{1}{2}\int_0^{L_0}\Big(u\mathcal{M}u+(\omega_0-1)\left(u^2+u\mathcal{M}u\right)\Big)dx\\
&=\frac{1}{2}\int_{0}^{L_0}\Big(\omega_0u\mathcal{M}u+(\omega_0-1)u^2\Big).
\end{split}
\]

\begin{lemma}\label{minlem}
For any $\gamma>0$, the minimization problem \eqref{infB} has at least one solution, that is, there exists $\phi\in Y_\gamma$ satisfying
$$
B(\phi)=\inf_{ u\in Y_{\gamma}}B(u).
$$
\end{lemma}
\begin{proof}
First of all note that, from $(\ref{A1A2})$,   $B$ is an equivalent norm in $X$, yielding $m\geq0$.
Let $\{u_n\}$ be a minimizing sequence for \eqref{infB}, that is, a sequence in $Y_\gamma$ satisfying 
$$\displaystyle B(u_n)\rightarrow\inf_{u\in Y_{\gamma}} B(u), \ \  \mbox{as} \ n\rightarrow \infty.$$ 

It is easy to check that
$$
\frac{(\omega_0-1)}{2}\|u_n\|_X\leq B(u_n)\leq \omega_0\|u_n\|_X, \ \ \mbox{for all} \ n\in\mathbb{N},
$$ 
implying that $\{u_n\}$ is bounded in $X$. Consequently, there exists 
$\phi\in X$  such that, up to a subsequence,
$$u_n\rightharpoonup \phi \ \ \mbox{weakly in} \ X,  \ \  \mbox{as} \ n\rightarrow \infty.$$

On other hand, using $m_1>1/3$ we get the energy space $X$ is compactly embedded in $L_{per}^{3}([0,L_0])$. Thus,
$$u_n\rightarrow \phi \ \ \mbox{in} \ L^3_{per}([0,L_0]),  \ \  \mbox{as} \ n\rightarrow \infty.$$
 Besides that, using the fact 
\begin{eqnarray}
\left|\int_{0}^{L_0}(u_n^3-\phi^3)dx\right|
&\leq&\int_{0}^{L_0}|u_n^3-\phi^3|dx \nonumber\\
&\leq&\|u_n-\phi\|_{L^3_{per}}^3+3\|u_n-\phi\|_{L^3_{per}}\|\phi\|_{L^3_{per}}\|u_n\|_{L^3_{per}}. \nonumber
\end{eqnarray}
we can say that $\int_0^{L_0}\phi^3 dx=\gamma$. 

 Moreover, thanks
to the weak lower semi-continuity of $B$, we have
\begin{equation*}
B(\phi)\leq\liminf_{n\rightarrow \infty} B(u_n)=m.
\end{equation*}
Therefore, $\phi$ satisfies (\ref{infB}).
\end{proof}

From Lemma \eqref{minlem} and  Lagrange's Multiplier Theorem, there exists of $C_1$ such that
$$\omega_0\mathcal{M}\phi+(\omega_0-1)\phi=C_1\phi^2.$$
We note that $\phi$ is nontrivial since $\gamma>0$. Furthermore, a simple scaling argument, gives us that $C_1$ can be chosen as $C_1=\frac{1}{2}$. Indeed, for $s\in\mathbb{R}$, 
\begin{eqnarray}
B(s\phi)&=&s^2B(\phi) \nonumber \\
&=&s^2\min_{u\in X}\left\{B(u);  \ \int_0^{L_0}u^3=\gamma\right\} \nonumber \\
&=&\min_{u\in X}\left\{B(su); \ \int_0^{L_0}u^3=\gamma\right\} \nonumber \\
&=&\min_{u\in X}\left\{B(u); \ \int_0^{L_0}u^3=s^3\gamma\right\}. \nonumber
\end{eqnarray}
Then, $\phi$ satisfies the equation
\begin{equation*}
\omega_0\mathcal{M}\phi+(\omega_0-1)\phi-\frac{1}{2}\phi^2=0.
\end{equation*}

\noindent In addition, we obtain that $\phi$ is smooth (by Proposition \ref{regularity}) and satisfies equation $(\ref{ode-wave})$ with $A_0=0$.

As before, let $\mathcal{L}_{0}=\omega_0\mathcal{M}+(\omega_0-1)-\phi$. Here, instead of assuming all assumption in $(H)$, we suppose the following:
 
 \vspace{0.4cm}
 
 \begin{enumerate}
 	\item[\textit{(H1)}] $\phi$ is even and $\ker(\mathcal{L}_{0})=[\phi']$.
 \end{enumerate}

 \vspace{0.4cm}
 
We note that, in view of (\ref{infB}), we have 
\begin{equation}\label{lpositive}
\mathcal{L}_{0}\big|_{\{R'(\phi)\}^{\bot}}\geq0
\end{equation}
where $R(u)=\int_0^{L_0}u^3dx$. So we have.

\begin{proposition}\label{Propn}
Let $\phi\in X$ be the local minimizer satisfying (\ref{infB}). Then $\mbox{n}\left(\mathcal{L}_0\right)=1$, where $\mbox{n}\left(\mathcal{L}_0\right)$ stands for the number of negative eigenvalues of $\mathcal{L}_0$ acting on $L^2_{per}([0,L_0])$.
\end{proposition}
\begin{proof}
Since 
$$\langle\mathcal{L}_0\phi,\phi\rangle=-\frac{1}{2}\int_0^{L_0}\phi^3dx=-
\int_0^{L_0}\left(\omega_0\phi\mathcal{M}\phi+(\omega_0-1)\phi^2\right)dx<0.$$
It follows that $\mathcal{L}_0$ acting on $L^2_{per}([0,L_0])$ must have at least one negative eigenvalue. Moreover, using (\ref{lpositive}) we have, by Courant's mini-max principle, that $\mathcal{L}_0$ has at most one negative eigenvalues. Therefore, ${n}\left(\mathcal{L}_0\right)=1$.
\end{proof}

%Straightforward calculations enable us to deduce $\mathcal{L}_{0}(1-\Psi)=(\omega_0-1)$ and $\mathcal{L}_{0}\phi=-\frac{1}{2}\phi^2$. Moreover, since $\omega_0>1$ and $(H2)$ is assumed, we get that $\{1,\phi,\phi^2\}$ belongs to the range of $\mathcal{L}_{0}$. 
%Now, using $(H2)$ and Proposition \ref{Propn}
%we obtain, from \cite[Proposition 3.1]{hur}, that $\ker(\mathcal{L}_{0})=[\phi']$.\\
%\indent Since $\phi'$ is odd and the kernel of $\mathcal{L}_{0}$ is simple, we can apply Theorem $\ref{existcurve}$ to obtain the existence of an open set $\mathcal{O}\subset(1,+\infty)\times (-\delta_0,\delta_0)$, $\delta_0>0$, and a smooth surface of $L_0-$periodic waves $\psi:=\phi_{(\omega,A)}$, $(\omega,A)\in\mathcal{O}$, which solves equation $(\ref{ode-wave})$.\\
%\indent We prove that $n(\mathcal{L})=1$, where $\mathcal{L}=\omega\mathcal{M}+(\omega-1)-\psi$, by using the Index Theorem in \cite[Theorem 5.3.2]{KP} with $\textbf{D}:=\langle \mathcal{L}_{\phi}^{-1}R'(\phi),R'(\phi)\rangle$, $R(u)=\int_0^{L_0}u^3dx$. Indeed, since the kernel of $\mathcal{L}_{0}$ is simple and generated by $\phi'$, one has that 
%$$\textbf{D}=-18\int_0^{L_0}\phi^3dx=-
%36\int_0^{L_0}\left(\omega\phi\mathcal{M}\phi+(\omega-1)\phi^2\right)dx<0.$$
%Since $\phi$ is a local minimizer, we see that $n\left(\mathcal{L}_{0}\big|_{\{R'(\phi)\}^{\bot}}\right)=0$ and, from Index Theorem, we obtain $n(\mathcal{L}_{0})=n(\textbf{D})=1$.

Since $\phi'$ is odd and the kernel of $\mathcal{L}_{0}$ is simple by $(H1)$, we can apply Theorem $\ref{existcurve}$ to obtain the existence of an open set $\mathcal{O}\subset(1,+\infty)\times (-\delta_0,\delta_0)$, $\delta_0>0$, and a smooth surface of $L_0-$periodic waves $\psi:=\phi_{(\omega,A)}$, $(\omega,A)\in\mathcal{O}$, which solves equation $(\ref{ode-wave})$. Proposition $\ref{prop1}$ can be used to conclude that the kernel of the linearized operator $\mathcal{L}=\omega\mathcal{M}+(\omega-1)-\psi$ is simple, generated by $\psi'$ and $n(\mathcal{L})=1$, for all pair $(\omega,A)\in\mathcal{O}\subset (1,+\infty)\times (-\delta_0,\delta_0)$.\\
\indent The next step is to calculate $s(\phi)$, where $\phi=\phi_{(\omega_0,0)}$, $\omega_0>1$. In fact, one has 
\begin{equation}\label{s123}s(\phi)=(2\omega_0(\omega_0-1)+1)M(\phi)+\omega_0\int_0^{L_0}\phi\mathcal{M}\phi dx+(1-\omega_0)L_0.
\end{equation}
We shall give a convenient expression for $M(\phi)$. In fact, from $(\ref{ode-wave})$ with $(\omega,A)=(\omega_0,0)$, we get

\begin{equation}\label{mean}
M(\phi)=\frac{1}{2(\omega_0-1)}\int_0^{L_0}\phi^2dx.
\end{equation}
On the other hand, multiplying equation $(\ref{ode-wave})$ by $\phi$ and integrating the result, one has

\begin{equation}\label{mean1}
\int_0^{L_0}\phi^2dx=\frac{1}{2(\omega_0-1)}\int_0^{L_0}\phi^3dx-\frac{\omega_0}{\omega_0-1}\int_0^{L_0}\phi\mathcal{M}\phi dx.
\end{equation}
Thus, from $(\ref{mean})$, $(\ref{mean1})$ and the fact that $\int_0^{L_0}\phi^3dx=\gamma$, we obtain
\begin{equation}\label{mean2}
M(\phi)=\frac{\gamma}{4(\omega_0-1)^2}-\frac{\omega_0}{2(\omega_0-1)^2}\int_0^{L_0}\phi\mathcal{M}\phi dx.
\end{equation}
Now, substituting the value of $M(\phi)$ in $(\ref{mean2})$ into $(\ref{s123})$, we deduce, after some calculations 
%and the fact that$\int_0^{L_0}\phi\mathcal{M}\phi dx>0$
\begin{equation}\label{sfinal}\begin{array}{lllll}
s(\phi)&=&\displaystyle\frac{(2\omega_0(\omega_0-1)+1)\gamma}{4(\omega_0-1)^2}+\frac{\omega_0(1-2\omega_0)}{2(\omega_0-1)^2}\int_0^{L_0}\phi\mathcal{M}\phi dx+(1-\omega_0)L_0.
\end{array}\end{equation}

To estimate the middle term on the right-hand side of \eqref{sfinal}, observe, from \eqref{mean1}, that
$$
\omega_0\int_0^{L_0}\phi\mathcal{M}\phi dx<\frac{\gamma}{2},
$$
from which we deduce
\begin{equation}\label{sfinal1}
\frac{\omega_0(1-2\omega_0)}{2(\omega_0-1)^2}\int_0^{L_0}\phi\mathcal{M}\phi dx>\frac{\gamma(1-2\omega_0)}{4(\omega_0-1)^2}.
\end{equation}
By replacing \eqref{sfinal1} into \eqref{sfinal} we then infer
\[
\begin{split}
s(\phi)&>\displaystyle\frac{(2\omega_0(\omega_0-1)+1)\gamma}{4(\omega_0-1)^2} + \frac{\gamma(1-2\omega_0)}{4(\omega_0-1)^2} + (1-\omega_0)L_0\\
&= \frac{\gamma}{2}-(\omega_0-1)L_0.
\end{split}
\]

Hence, as an application of Theorem \ref{coro14} we just have proved the following.

\begin{theorem}\label{mintheo}
Let $L_0>0$ and $\omega_0>1$ be fixed. Choose $\gamma>0$ such that
$$
\gamma>2(\omega_0-1)L_0.
$$
Let $\phi\in Y_\gamma$ be a minimizer of problem \eqref{infB} according to Lemma \ref{minlem} and assume that $(H1)$ holds. Then $\phi$ is orbitally stable in $X$ by the periodic flow of \eqref{rDE}.
\end{theorem}

\begin{obs}
The arguments above, assumption $(H1)$, Corollary $\ref{coro6789}$ and the smoothness of the involved functions are sufficient to deduce the orbital stability of the smooth surface of periodic waves $\phi_{(\omega,A)}$, $(\omega,A)\in\widetilde{\mathcal{O}}\subset\mathcal{O}\subset (1,+\infty)\times(-\delta_0,\delta_0)$, obtained from $\phi$.
\end{obs}

\indent As an application of the approach presented in this subsection, we are going to use Theorem \ref{mintheo} to get the orbital stability of periodic waves of the model in $(\ref{BO1})$. Our intention is to give a considerable simplification of the arguments in \cite{ASB}. In fact, to get assumption $(H)$, the authors have used the expansion in Fourier series of the explicit periodic wave which solves the equation $(\ref{ode-wave})$ when $\mathcal{M}=\mathcal{H}\partial_x$ and $A=0$. Moreover, it has been required in \cite{ASB} an explicit calculation of the derivative in terms of $\omega$ of the inner product $\langle\phi,\phi+\mathcal{H}\phi'\rangle$ to conclude the orbital stability.\\
\indent To simplify the notation, let us consider $L_0=4\pi$. The minimizer $\phi$ obtained in Lemma $\ref{minlem}$ solves the equation 
\begin{equation}
	\omega_0\mathcal{H}\phi'
	+(\omega_0-1)\phi-\frac{1}{2}\phi^2=0.
	\label{ode-BO}
\end{equation} 
Let $\omega_0>2$ be fixed. Using similar arguments as those in \cite{ABS}, we get an explicit solution as
\begin{equation}\label{explBO}
	\phi(x)=\omega_0\left(\frac{\sinh(\eta)}{\cosh(\eta)-\cos\left(\frac{ x }{2}\right)}\right),
\end{equation}
where $\tanh(\eta)=\frac{\omega_0}{(\omega_0-1)2}$.

\indent On the other hand, since $\omega_0>2$ is arbitrary, we deduce from $(\ref{explBO})$ that $\phi$ can be seen as a curve depending smoothly on $\omega\in(2,+\infty)$ and this fact is a cornestone to conclude assumption $(H1)$. In fact, clearly $\phi$ in $(\ref{explBO})$ is even. Let us consider $\widetilde{\mathcal{L}_0}=\frac{1}{\omega_0}\mathcal{L}_0$, thus
\begin{equation}
	\label{operaL1}
	\widetilde{\mathcal{L}_0}\left(\frac{1}{\omega_0}+\frac{\phi}{\omega_0}-\frac{\partial\phi}{\partial\omega}\Big|_{\omega=\omega_0}\right)=\frac{1}{\omega_0}\left(1-\frac{1}{\omega_0}\right)\neq0.\end{equation}
\indent  Now, since $\widetilde{\mathcal{L}_0}(\phi)=-\frac{1}{2\omega_0}\phi^2$ and $\widetilde{\mathcal{L}_0}\left(\frac{\partial\phi}{\partial\omega}\Big|_{\omega=\omega_0}\right)=-\frac{1}{\omega_0^2}\phi-\frac{1}{2\omega_0^2}\phi^2$, we have from $(\ref{operaL1})$ that $\{1,\phi,\phi^2\}\subset \rm{Range}(\mathcal{L}_0)$. Proposition 3.2 in \cite{hur} gives us that $\ker(\mathcal{L}_0)=[\phi']$ as required in $(H1)$.\\
\indent The Poincar\'e-Wirtinger inequality applied to $\int_{-2\pi}^{2\pi}\phi\mathcal{H}\phi'dx$ combined with the equation $(\ref{ode-BO})$ give us
\begin{equation}\label{poinc}
	\int_{-2\pi}^{2\pi}\phi^3dx\geq 2\omega_0(\omega_0-1)\int_{-2\pi}^{2\pi}\phi dx+4(\omega_0-1)^2\int_{-2\pi}^{2\pi}\phi dx-\frac{\omega_0}{4\pi}\left(\int_{-2\pi}^{2\pi}\phi dx\right)^2.
\end{equation}
Last term in $(\ref{poinc})$ can be handled employing H\"older inequality to get, again from equation $(\ref{ode-wave})$ in this particular case, that
\begin{equation}\label{poinc1}\begin{array}{lllll}
		\displaystyle\int_{-2\pi}^{2\pi}\phi^3dx&\geq& \displaystyle 2\omega_0(\omega_0-1)\int_{-2\pi}^{2\pi}\phi dx+4(\omega_0-1)^2\int_{-2\pi}^{2\pi}\phi dx-2\omega_0(\omega_0-1)\int_{-2\pi}^{2\pi}\phi dx.\\\\
		&=&\displaystyle4(\omega_0-1)^2\int_{-2\pi}^{2\pi}\phi dx\end{array}
\end{equation}

\indent Since $\int_{-2\pi}^{2\pi}\phi dx=4\pi\omega_0$, we obtain by $(\ref{poinc1})$
\begin{equation}\label{poinc12}
	\int_{-2\pi}^{2\pi}\phi^3dx\geq 16\pi\omega_0(\omega_0-1)^2.
\end{equation}
Finally, it is easy to see that $\omega_0\in(2,+\infty)$ implies   
$$\gamma\geq 16\pi\omega_0(\omega_0-1)^2>8\pi(\omega_0-1),$$
and, according with Theorem $\ref{mintheo}$ one has the orbital stability of $\phi$.

\begin{obs}
In the general fractional case, that is, $\mathcal{M}=\Lambda^{\alpha}$, $\alpha\in (1/3,2]$ and $\Lambda=\sqrt{-\partial_x^2}$, assumption $(H1)$ holds (see \cite{hur}) since it is assumed the existence of a smooth surface of periodic waves, $(\omega,A)\in\widetilde{\mathcal{O}}\mapsto\phi_{(\omega,A)}\in H_{per}^n([0,L_0])$, $n\in\mathbb{N}$, which solves equation $(\ref{ode-wave})$ having fixed period $L_0>0$. The existence of such smooth surface prevents the existence of "fold points", that is, values of $(\omega,A)\in\widetilde{\mathcal{O}}$ such that $\mathcal{L}_{(\omega,A)}g=0$, for some $g\in D(\mathcal{L}_{(\omega,A)})$. In fact, the existence of a smooth surface of periodic waves solving equation $(\ref{ode-wave})$ enables us to deduce the existence of $\beta\in D(\mathcal{L}_{(\omega,A)})$ such that $\mathcal{L}\beta=1$, and thus, after a straightforward calculation one has $\{1,\phi_{(\omega,A)},\phi_{(\omega,A)}^2\}\subset \rm{Range}(\mathcal{L}_{(\omega,A)})$. This property can be combined with Proposition 3.2 in \cite{hur} to get the non-degeneracy of $\ker(\mathcal{L}_{(\omega,A)})$ (see \cite{natali-le-peli} for details). By taking a  $A$ small enough and $\omega>1$, one sees that $\phi_{(\omega,A)}$ is orbitally stable in $H_{per}^{\alpha/2}([0,L_0])$ and, therefore, we can conclude that equation $(\ref{rDE})$, in the fractional case, always admits stable periodic waves. However, our approach diverges, in some sense the arguments in \cite{hur}, because, in this case, it was not necessary to calculate the signal of the Hessian matrix associated to the conserved quantities $F$ and $M$ in $(\ref{Fu})$ and $(\ref{Mu})$.

\end{obs}

\section*{Acknowledgements}

F. C. is supported by FAPESP/Brazil grant 2017/20760-0. F. N. is partially supported by CNPq/Brazil and Funda\c{c}\~ao Arauc\'aria/Brazil grants 304240/2018-4 and 002/2017.  A. P. is partially supported by CNPq/Brazil grants 402849/2016-7 and 303098/2016-3. The second author would like to express his gratitude to McMaster University for its hospitality and Dmitry E. Pelinovsky for fruitful comments regarding this work.

\end{document}